\documentclass[12pt,leqno,draft]{article} 
\usepackage{amsmath,amssymb,amsthm,amsfonts,bm}
\usepackage{enumerate,color}
\makeatletter
\def\@cite#1#2{[{{\bfseries #1}\if@tempswa , #2\fi}]}
\renewcommand{\section}{%
\@startsection{section}{1}{\z@}
{0.5truecm plus -1ex minus -.2ex}%
{1.0ex plus .2ex}{\bfseries\large}}
\def\@seccntformat#1{\csname the#1\endcsname.\ }
\makeatother

\numberwithin{equation}{section} 
\newtheorem{thm}{Theorem}[section]

\newtheorem{lem}[thm]{Lemma}

\theoremstyle{definition}
\newtheorem{df}{Definition}[section]
\newtheorem{remark}{Remark}[section]
\newtheorem*{ex1}{Example 6.1 (porous media and Cahn--Hilliard type equations)}
\newtheorem*{ex2}{Example 6.2 (fast diffusion and Cahn--Hilliard type equations)}
\newtheorem*{prth1.1}{Proof of Theorem 1.1}
\newtheorem*{prth1.2}{Proof of Theorem 1.2}
\newtheorem*{prth1.3}{Proof of Theorem 1.3}
\newtheorem*{prth5.1}{Proof of Theorem 5.1}

\newcommand{\ep}{\varepsilon}
\newcommand{\pa}{\partial}
\newcommand{\RN}{\mathbb{R}^N}

\begin{document}
\footnote[0]
    {2010{\it Mathematics Subject Classification}\/. 
    Primary: 35K59, 35K35; Secondary: 47H05.
    }
\footnote[0]
    {{\it Key words and phrases}\/: 
    quasilinear parabolic equations; porous media type equations; 
    Cahn--Hilliard type systems; 
    subdifferential operators.
    }
\begin{center}
    \Large{{\bf A direct approach to quasilinear parabolic equations 
\\
on unbounded domains by Br\'ezis's theory for 
\\
subdifferential operators
           }}
\end{center}
\vspace{5pt}
\begin{center}
    Shunsuke Kurima\\
    \vspace{2pt}
    Department of Mathematics, 
    Tokyo University of Science\\
    1-3, Kagurazaka, Shinjuku-ku, Tokyo 162-8601, Japan\\
    {\tt shunsuke.kurima@gmail.com}\\
    \vspace{12pt}
    Tomomi Yokota%
   \footnote{Corresponding author}%
   \footnote{Partially supported by Grant-in-Aid for
    Scientific Research (C), No.\,16K05182.}\\
    \vspace{2pt}
    Department of Mathematics, 
    Tokyo University of Science\\
    1-3, Kagurazaka, Shinjuku-ku, Tokyo 162-8601, Japan\\
    {\tt yokota@rs.kagu.tus.ac.jp}\\
    \vspace{2pt}
\end{center}
\begin{center}    
    \small \today
\end{center}

\vspace{2pt}
\newenvironment{summary}
{\vspace{.5\baselineskip}\begin{list}{}{%
     \setlength{\baselineskip}{0.85\baselineskip}
     \setlength{\topsep}{0pt}
     \setlength{\leftmargin}{12mm}
     \setlength{\rightmargin}{12mm}
     \setlength{\listparindent}{0mm}
     \setlength{\itemindent}{\listparindent}
     \setlength{\parsep}{0pt}
     \item\relax}}{\end{list}\vspace{.5\baselineskip}}
\begin{summary}
{\footnotesize {\bf Abstract.} 
    This paper is concerned with existence and uniqueness 
    of solutions to two kinds of quasilinear parabolic equations. 
    One is described as the following form which includes 
    the porous media and fast diffusion type equations:  
    \begin{equation}
        \frac{\partial u}{\partial t} + (-\Delta+1)\beta(u) 
        = g \quad \mbox{in}\ \Omega\times(0, T) \tag*{(E)}\label{E}
    \end{equation}
    and the other is the Cahn--Hilliard type system summarized as
    \begin{equation}
    \frac{\partial u_{\ep}}{\partial t} 
        + (-\Delta+1)(\ep(-\Delta+1)u_{\ep} + \beta(u_{\ep}) + \pi_{\ep}(u_{\ep}))
        = g \quad \mbox{in}\ \Omega\times(0, T), \tag*{(E)$_{\ep}$}\label{Eep}
    \end{equation}
    where $\Omega\subset\RN$ is an unbounded domain 
    with smooth bounded boundary, $N\in{\mathbb N}$, $T>0$, 
    $\beta$ is a single-valued maximal monotone function on 
    $\mathbb{R}$, e.g., $\beta(r) = |r|^{q-1}r + r\ (q >0)$  
    and $\pi_{\ep}$ is an anti-monotone function on $\mathbb{R}$, 
    e.g., $\pi_{\ep}(r) = -\ep r\ (\ep>0)$.    
    In the case that $N = 2, 3$, $\Omega$ is {\it bounded} and 
    $-\Delta + 1$ is replaced with $-\Delta$,     
    existence of solutions to \ref{E} was already proved 
    by Br\'ezis's theory for subdifferential operators. 
    On the other hand, it is known that existence of solutions 
    to \ref{Eep} is obtained from an approach via its approximate 
    problem whose solvability is proved by applying 
    an abstract theory for doubly nonlinear evolution inclusions; 
    however, the proof is based on compactness methods 
    and hence the case of unbounded domains is excluded 
    from the framework. The present paper applies Br\'ezis theory 
    directly to both \ref{E} and \ref{Eep} and gives 
    existence results for these two equations even if 
    $\Omega$ is {\it unbounded}. Moreover, an error estimate 
    between \ref{E} and \ref{Eep} as in Colli and Fukao \cite{CF-2016} 
    is also proved via apriori estimates obtained directly. } 
\end{summary}
\newpage

\section{Introduction} \label{Sec1}
\subsection{Two problems}

We consider applications of 
Br\'ezis's theory for subdifferential operators proposed in \cite{Brezis}
to quasilinear parabolic equations on {\it unbounded} domains.
In \cite[Theorem 3.6]{Brezis} it is explained that
there exists a unique solution of the following Cauchy problem 
for abstract evolution equations:
\[
\begin{cases}
u'(t) + \partial\psi(u(t)) \ni \tilde{f}(t) 
         \quad \mbox{in}\ X \quad \mbox{for a.a.}\ t\in(0, T),
         \\[2mm]
         u(0) = u_0 \quad \mbox{in}\ X,
\end{cases}
\]
where $X$ is a Hilbert space, $\partial\psi$ is a subdifferential 
operator of a proper lower semicontinuous convex function $\psi$, 
$u: [0, T] \to X$ is an unknown function and $\tilde{f} \in L^2(0, T; X)$ is a given function. 
The theory is often applied to problems on {\it bounded} domains 
(see some examples given in \cite{Brezis}). 
The theme of this paper is to apply the theory in 
\cite{Brezis} {\it directly} to two 
quasilinear parabolic partial differential equations 
on {\it unbounded} domains.

%
%

The first purpose is that we apply the above Br\'ezis's theory to show 
existence and uniqueness of solutions to the following problem: 
 \begin{equation*}\tag*{(P)}\label{P}
     \begin{cases}
         \dfrac{\partial u}{\partial t}+(-\Delta+1)\beta(u) = g  
         & \mbox{in}\ \Omega\times(0, T),
 \\[3mm]
         \partial_{\nu}\beta(u) = 0                                   
         & \mbox{on}\ \partial\Omega\times(0, T),
 \\[3mm]
        u(0) = u_0                                         
         & \mbox{in}\ \Omega, 
     \end{cases}
 \end{equation*}
where $\Omega$ is an {\it unbounded} domain 
in $\RN$ with smooth bounded boundary $\partial\Omega$, 
$N \in{\mathbb N}$, $T>0$, $g, u_{0}$ is given functions,  
and
$\pa_\nu$ denotes differentiation with respect to
the outward normal of $\pa\Omega$. 
If $N = 2, 3$, $\Omega$ is bounded and $-\Delta+1$ is replaced with $-\Delta$, then 
\ref{P} represents the porous media equation 
(see, e.g., \cite{ASS-2016, M-2010, V-2007, Y-2008}), 
the Stefan problem 
(see, e.g., \cite{BP-2005, D-1977, Fri-1968, F-2016, HK-1991}), 
the fast diffusion equation 
(see, e.g., \cite{B-1983, RV-2002, V-2007}), etc. 
In this case, existence and uniqueness of solutions to these problems can be proved by a direct application of \cite{Brezis}. However, since the proof of 
the existence depends on boundedness of $\Omega$, there seems to be no work on the problem on unbounded domains via \cite{Brezis}. In this paper we mainly study the case 
such as $\beta (u)= |u|^{q-1}u + u$  $(q>1)$.

%
%

The second purpose is to show that the theory in \cite{Brezis} is 
directly applicable to 
the following problem for the Cahn--Hilliard type system:
\begin{equation*}\tag*{(P)$_{\ep}$}\label{Pep}
\begin{cases}
\dfrac{\partial u_{\ep}}{\partial t} 
+ (-\Delta+1)\mu_{\ep} = 0 
& \mbox{in}\ \Omega\times(0, T),
\\[3mm] 
\mu_{\ep}=\ep(-\Delta+1)u_{\ep} 
+ \beta(u_{\ep})+\pi_{\ep}(u_{\ep})-f 
& \mbox{in}\ \Omega\times(0, T), 
\\[3mm]
\partial_{\nu}{\mu_{\ep}} = \partial_{\nu}{u_{\ep}} = 0 
& \mbox{on}\ \partial\Omega\times(0, T),
\\[3mm]
u_{\ep}(0) = u_{0\ep} 
& \mbox{in}\ \Omega,
\end{cases}
\end{equation*}
where $\pi_{\ep}$ is an anti-monotone function 
with 
$\ep>0$,  
$f$ is a function determined by $g$ and  
$u_{0\ep}$ is a given function. 
If $N = 2, 3$, $\Omega$ is bounded and $-\Delta+1$ is 
reduced to $-\Delta$, then \ref{Pep} represents 
the Cahn--Hilliard system (see e.g., \cite{CH-1958, CF-2015, EZ-1986}) 
and is regarded as an approximate problem to \ref{P} 
(see \cite{CF-2016, F-2016}). In particular, 
in the proof of existence of solutions to the problem 
in \cite{CF-2016}, {\it one more approximation} (P)$_{\ep, \lambda}$ of 
\ref{Pep} was essentially required,  
where existence of solutions to (P)$_{\ep, \lambda}$ was 
proved by applying the abstract theory by Colli and Visintin \cite{CV-1990} 
for doubly nonlinear evolution inclusions of the form  
    $$
    Au'(t) + \partial\psi(u(t)) \ni k(t)
    $$
 with some bounded monotone operator $A$ and 
 some proper lower semicontinuous convex function $\psi$. 
 Since the theory is based on compactness methods, 
 boundedness of $\Omega$ is necessary and 
 hence the case of unbounded domains is excluded from 
 their frameworks. 

 The relation between \ref{P} and \ref{Pep} was recently 
 studied by Colli and Fukao \cite{CF-2016} in the case 
 stated above. More precisely, in \cite{CF-2016},  
 existence of weak solutions to \ref{P} and \ref{Pep} with error 
 estimates was established under the condition that $N$ = $2, 3$, 
 $\Omega$ is a bounded domain with smooth boundary 
 and $-\Delta + 1$ is replaced with $-\Delta$ in 
 \ref{P} and \ref{Pep}. 
 In particular, they considered the case of degenerate diffusion and their approach to     degenerate diffusion equations from the Cahn--Hilliard system made a new development. 
  They established the error estimate that 
 the solution of \ref{Pep} converges to solution of \ref{P} 
 in the order $\ep^{1/2}$ as $\ep \searrow 0$. 
 Their proof was also based 
 on one more approximation (P)$_{\ep, \lambda}$,  
 while in this paper we will directly establish an error estimate 
 without using (P)$_{\ep, \lambda}$.

\subsection{Main result for \ref{P}}\label{Sec1.2}

 %
 %
 %

 Before stating the main result for \ref{P}, we give some conditions, 
 notations and definitions. 
 We will assume that $\beta$, $g$, $f$, and $u_{0}$ satisfy the following conditions: 
%
%
%
 \begin{enumerate} 
 \item[(C1)] $\beta : \mathbb{R} \to \mathbb{R}$                                
 is a single-valued maximal monotone function 
 and $\beta(r) = \hat{\beta}\,'(r) = \partial\hat{\beta}(r)$, where 
 $\hat{\beta}\,'$ and $\partial\hat{\beta}$ are the differential and  
 subdifferential 
 of a proper differentiable (lower semicontinuous) convex function 
 $\hat{\beta} : \mathbb{R} \to [0, +\infty]$ 
 satisfying $\hat{\beta}(0) = 0$. 
 This entails $\beta(0) = 0$. 
 There exists a constant $c_{1} > 0$ such that 
   \begin{equation*}
   \hat{\beta}(r) \geq c_{1}|r|^2 
   \quad \mbox{for all}\ r\in\mathbb{R}.
   \end{equation*}
 For all $z \in L^2(\Omega)$, if $\hat{\beta}(z) \in L^1(\Omega)$,  
 then $\beta(z) \in L_{{\rm loc}}^1(\Omega)$. 
 Moreover, 
 for all $z \in L^2(\Omega)$ and 
 for all $\psi \in C_{\mathrm{c}}^{\infty}(\Omega)$, 
 if $\hat{\beta}(z)\in L^1(\Omega)$,  
 then $\hat{\beta}(z + \psi) \in L^1(\Omega)$.
 \item[(C2)] $g\in L^2\bigl(0, T; L^2(\Omega)\bigr)$. 
 Then we fix a solution 
 $f\in L^2\bigl(0, T; H^2(\Omega)\bigr)$ of 
 \begin{equation}\label{a3eq}
     \begin{cases}
         (-\Delta+1)f(t) = g(t) 
         & \mbox{a.e.\ in}\ \Omega,
     \\[2mm] 
         \partial_{\nu}f(t) = 0 
         & \mbox{in the sense of traces on}\ \partial\Omega 
     \end{cases}
 \end{equation}
 for a.a.\ $t\in(0, T)$, that is, 
 \begin{equation}
 \int_{\Omega}\nabla f(t)\cdot\nabla z + 
 \int_{\Omega}f(t)z\ = \int_{\Omega}g(t)z \quad 
 \mbox{for all}\ z\in H^1(\Omega). 
 \end{equation}
 \item[(C3)] $u_0 \in L^2(\Omega)$ and $\hat{\beta}(u_0) \in L^1(\Omega)$.
 \end{enumerate}
 
 We put the Hilbert spaces 
   \begin{equation}\label{HV}
   H:=L^2(\Omega), \quad V:=H^1(\Omega)
   \end{equation}
 with inner products $(\cdot, \cdot)_H$ 
 and $(\cdot, \cdot)_V$, respectively. 
 Moreover, we use 
   \begin{equation}\label{W}
   W:=\bigl\{z\in H^2(\Omega)\ |\ \partial_{\nu}z = 0 \quad 
   \mbox{a.e.\ on}\ \partial\Omega\bigr\}.
   \end{equation}
 The notation $V^{*}$ denotes the dual space of $V$ with 
 duality pairing $\langle\cdot, \cdot\rangle_{V^*, V}$. 
 Moreover, we define a bijective mapping $F : V \to V^{*}$ and 
 the inner product in $V^{*}$ as 
    \begin{align}
    &\langle Fv_{1}, v_{2} \rangle_{V^*, V} := 
    (v_{1}, v_{2})_{V} \quad \mbox{for all}\ v_{1}, v_{2}\in V, 
    \label{defF}
    \\[1mm]
    &(v_{1}^{*}, v_{2}^{*})_{V^{*}} := 
    \left\langle v_{1}^{*}, F^{-1}v_{2}^{*} 
    \right\rangle_{V^*, V} 
    \quad \mbox{for all}\ v_{1}^{*}, v_{2}^{*}\in V^{*};
    \label{innerVstar}
    \end{align}
 note that $F : V \to V^{*}$ is well-defined by 
 the Riesz representation theorem. 
 We remark that (C2) implies 
 \begin{equation}\label{Ffg}
 Ff(t)=g(t) \quad \mbox{for a.a.}\ t\in(0, T).
 \end{equation}
 
 We define weak solutions of \ref{P} as follows.
%
%
%
 \begin{df}         
 A pair $(u, \mu)$ with 
    \begin{align*}
    &u\in H^1(0, T; V^{*})\cap L^{\infty}(0, T; H), 
    \\
    &\mu\in L^2(0, T; V)
    \end{align*}
 is called a {\it weak solution} of \ref{P} 
 if $(u, \mu)$ satisfies 
    \begin{align}
        & \bigl\langle u'(t), z\bigr\rangle_{V^{*}, V} + 
          \bigl(\mu(t), z\bigr)_{V} = 0 \quad 
          \mbox{for all}\ z \in V\ \mbox{and a.a}.\ t\in(0, T), 
          \label{de4}
     \\[3mm]
        & \mu(t) = \beta(u(t)) - f(t) \quad \mbox{in}\ V 
          \quad \mbox{for a.a.}\ t\in(0, T), \label{de5}
     \\[3mm]
        & u(0) = u_{0} \quad \mbox{a.e.\ on}\ \Omega. \label{de6}
     \end{align}
 \end{df}

Now the main result for \ref{P} reads as follows.
 \begin{thm}\label{maintheorem1}
 Assume {\rm (C1)-(C3)}. 
 Then there exists a unique weak solution $(u, \mu)$ of {\rm \ref{P}}, 
 satisfying                                                 
     \begin{equation*}
        u \in H^1(0, T; V^{*})\cap L^{\infty}(0, T; H),
         \quad \mu \in L^2(0, T; V).  
     \end{equation*}
 Moreover, for all $t\in[0, T]$,
     \begin{align}
     &\int_{0}^{t}\bigl|u'(s)\bigr|_{V^{*}}^2\,ds 
          + 2c_{1}|u(t)|_{H}^2 
     \leq M_{1}, \label{es1} 
     \\
     &\int_{0}^{t}|\mu(s)|_{V}^2\,ds \leq M_{1}, \label{es2} 
     \\
     &\int_{0}^{t}|\beta(u(s))|_{V}^2\,ds 
     \leq 2\bigl(M_{1}+|f|_{L^2(0, T; V)}^2 \bigr) \label{es3}, 
     \end{align}
 where $M_1:=2\int_{\Omega}\hat{\beta}(u_0) + |f|_{L^2(0, T; V)}^2$.
 \end{thm}
\subsection{Main result for \ref{Pep}}

We will assume that $\pi_{\ep}$ and $u_{0\ep}$ satisfy the following conditions:
\begin{enumerate}
\item[(C4)] $\pi_{\ep} : \mathbb{R} \to \mathbb{R}$ is                          
 a Lipschitz continuous function and $\pi_{\ep}(0) = 0$ 
 for all $\ep\in(0, 1]$. 
 Moreover, there exists a constant $c_{2}(\ep)>0$ depending on $\ep$ 
 such that there exists $\overline{\ep} \in (0, 1]$ satisfying 
 $c_2(\ep) < 2c_1$ for all $\ep \in (0, \overline{\ep}]$ and 
   \begin{equation}\label{a3ineq}
   \bigl|\pi_{\ep}'\bigr|_{L^{\infty}(\mathbb{R})} 
   \leq c_{2}(\ep)
   \quad \mbox{for all}\ \ep\in(0, 1].
   \end{equation} 
 Moreover, $r\mapsto\frac{\ep}{2}r^2+\hat{\pi_{\ep}}(r)$ 
 is convex, where 
 $\hat{\pi_{\ep}}(r):=\int_{0}^{r}\pi_{\ep}(s)\,ds$.
 \item[(C5)] 
 Let $u_{0\ep}\in H^1(\Omega)$ 
 fulfill $\hat{\beta}(u_{0\ep})\in L^1(\Omega)$ 
 and 
   \begin{equation}\label{a4ineq}
   |u_{0\ep}|_{L^2(\Omega)}^2 \leq c_{3}(\ep),\quad 
   \int_{\Omega}\hat{\beta}(u_{0\ep}) \leq c_{3}(\ep), \quad 
   \ep|u_{0\ep}|_{H^1(\Omega)}^2 \leq c_{3}(\ep),
   \end{equation}
 where $c_{3}(\ep) > 0$ is a constant depending on $\ep$.
\end{enumerate}
%
%
%
%
Let $H$, $V$ and $W$ be as in Section \ref{Sec1.2}.
Then we define weak solutions of \ref{Pep} as follows.
 \begin{df}        
 A pair $(u_{\ep}, \mu_{\ep})$ with 
    \begin{align*}
    &u_{\ep}\in H^1(0, T; V^{*})\cap L^{\infty}(0, T; V)\cap 
    L^2(0, T; W), 
    \\
    &\mu_{\ep}\in L^2(0, T; V)
    \end{align*}
 is called a {\it weak solution} of \ref{Pep} if 
 $(u_{\ep} ,\mu_{\ep})$ 
 satisfies 
    \begin{align}
        & \bigl\langle u_{\ep}'(t), z\bigr\rangle_{V^{*}, V} + 
          \bigl(\mu_{\ep}(t), z\bigr)_{V} = 0 \quad 
          \mbox{for all}\ z \in V\ 
          \mbox{and a.a}.\ t\in(0, T), \label{de7}
     \\[3mm]
        & \mu_{\ep}(t) = \ep(-\Delta + I)u_{\ep}(t) + 
        \beta(u_{\ep}(t)) + \pi_{\ep}(u_{\ep}(t)) - f(t) 
        \quad \mbox{in}\ V
        \quad \mbox{for a.a.}\ t\in(0, T), \label{de8}
     \\[3mm]
        & u_{\ep}(0) = u_{0\ep} 
        \quad \mbox{a.e.\ on}\ \Omega. \label{de9}
     \end{align}
 \end{df}
Now the main result for \ref{Pep} reads as follows.
%
 \begin{thm}\label{maintheorem2}
 Assume {\rm (C1)-(C5)}. Then there exists $\overline{\ep}\in(0, 1]$ such that 
 for every $\ep\in(0, \overline{\ep}]$                                    
 there exists a unique weak solution $(u_{\ep}, \mu_{\ep})$ of 
 {\rm \ref{Pep}}, satisfying
     \begin{equation*}
        u_{\ep} \in H^1(0, T; V^{*})
        \cap L^{\infty}(0, T; V)\cap L^2(0, T; W), 
        \quad \mu_{\ep} \in L^2(0, T; V). 
     \end{equation*} 
 Moreover, for all $t\in[0, T]$ and $\ep\in(0, \overline{\ep}]$, 
     \begin{align}
     &\int_{0}^{t}\bigl|u_{\ep}'(s)\bigr|_{V^{*}}^2\,ds + 
     \ep|u_{\ep}(t)|_{V}^2 + \bigl(2c_1- c_2(\ep)\bigr)|u_{\ep}(t)|_{H}^2
     \leq M_{2}(\ep), \label{epes1}
     \\
     &\int_{0}^{t}|\mu_{\ep}(s)|_{V}^2\,ds \leq M_{2}(\ep), 
     \label{epes2}
     \\
     &\int_{0}^{t}|\beta(u_{\ep}(s))|_{H}^2\,ds 
     \leq 3\left(M_{2}(\ep) + \frac{c_2(\ep)^2M_{2}(\ep)T}{2c_1- c_2(\ep)}
                                                   + |f|_{L^2(0, T; V)}^2 \right) \label{epes3}, 
     \\
     &\int_{0}^{t}|\ep u_{\ep}(s)|_{W}^2\,ds 
     \leq 16L^2\left(M_{2}(\ep) + \frac{c_2(\ep)^2M_{2}(\ep)T}{2c_1- c_2(\ep)} 
                                                      + |f|_{L^2(0, T; V)}^2 \right) \label{epes4}, 
     \end{align}
 where $M_{2}(\ep):=3c_3(\ep) + c_2(\ep)c_3(\ep) + |f|_{L^2(0, T; V)}^2$ 
 and L is a positive constant appearing in 
 the elliptic regularity estimate 
 $|w|_{W} \leq L|(-\Delta + I)w|_{H}$ for all $w \in W$.
 \end{thm}

%

\subsection{Outline of this paper}

 The strategy in the proofs of the main theorems is as follows. 
 As to Theorem \ref{maintheorem1}, 
 by setting a proper lower semicontinuous 
 convex function $\phi$ well, we can rewrite \ref{P} as an
 abstract nonlinear evolution equation with simple form 
 by the subdifferential of $\phi$:
    $$
    u'(t) + \partial\phi(u(t)) = g(t) \quad \mbox{in}\ V^{*},
    $$
 so that we can solve \ref{P} even on unbounded domains directly with 
 monotonicity methods (Lemma \ref{pre3}).  
 Moreover, from this, 
 Colli and Fukao \cite{CF-2016} proved apriori estimates for 
 solutions of \ref{P} by the limit of apriori estimates 
 for solutions of \ref{Pep} as $\ep \searrow 0$, while 
 we can obtain apriori estimates for 
 solutions of \ref{P} 
 directly. 
 The proof of Theorem \ref{maintheorem2} is parallel to that of 
 Theorem \ref{maintheorem1}, and hence we need not consider 
 one more approximation problem (P)$_{\ep, \lambda}$ 
 which cannot be used when $\Omega$ is unbounded. 
 In Theorem \ref{maintheorem3} 
 we can establish an error estimate between 
 the solution of \ref{P} and the solution of \ref{Pep} 
 without one more approximation of \ref{Pep}  
 even on unbounded domains.

 This paper is organized as follows. 
 In Section \ref{Sec2} 
 we give the definition and basic results for 
 subdifferentials of 
 proper lower semicontinuous convex functions 
 and useful results for 
 proving the main theorems.
 Sections \ref{Sec3} and \ref{Sec4} are devoted to the proofs of 
 Theorems \ref{maintheorem1} and \ref{maintheorem2}. 
 In Section \ref{Sec5} we prove an error estimate 
 between the solution of \ref{P} and the solution of \ref{Pep}.
 In Section \ref{Sec6} we give examples similar to  
 the porous media and the fast diffusion equations. 

\section{Preliminaries}\label{Sec2}

 We first give the definition and 
 basic results for subdifferentials of convex functions.
 \begin{df}\label{predf}
 Let $X$ be a Hilbert space. Given 
 a proper lower semicontinuous (l.s.c.\ for short) 
 convex function
 $\phi:X \to \overline{\mathbb{R}},$                                            
 the mapping $\partial\phi: X \to X$ defined by
    \begin{equation*}
      \partial\phi(z) := \bigl\{\tilde{z} \in X \ 
      |\ (\tilde{z}, w-z)_{X} \leq \phi(w) - \phi(z) 
      \quad \mbox{for all} \ w \in X \bigr\}
    \end{equation*} 
 is called the {\it subdifferential} operator of $\phi$, 
 with domain 
 $D(\partial\phi)
 :=\{z \in X\ |\ \partial\phi(z)\neq\emptyset\}$. 
 \end{df}
 The following lemma 
 is well-known (see e.g., Barbu \cite[Theorem 2.8]{Barbu1}).
 \begin{lem}\label{pre1}                                                                   
 Let $X$ be a Hilbert space and 
 let $\phi:X \to \overline{\mathbb{R}}$ 
 be a proper l.s.c.\ convex function. 
 Then $\partial\phi$ 
 is maximal monotone in $X$.
 \end{lem}
 The next asserts the chain rule. 
 For the proof see e.g., Showalter \cite[Lemma IV.4.3]{S-1997}.
 \begin{lem}\label{pre2}
 Let $\psi : X \to \overline{\mathbb{R}}$ be a proper, convex                           
 and l.s.c.\ function on a Hilbert space $X$. 
 If $u \in H^1(0,T; X)$ and there exists $v \in L^2(0,T; X)$ 
 such that $v \in \partial\psi(u)$ a.e.\ on $[0, T]$, 
 then the function $\psi \circ u$ is absolutely 
 continuous on $[0, T]$ and 
    \begin{equation*}
    \frac{d}{dt} \psi(u(t)) 
    = \bigl(w(t), u'(t)\bigr)_{X} 
    \quad \mbox{for a.a.}\ t\in[0, T]
    \end{equation*}
 for any function $w$ satisfying 
 $w(t) \in \partial\psi(u(t))$ for a.a.\ $t\in[0, T]$.
 \end{lem}

The following lemma plays a key role in the direct proof 
of existence of solutions 
to \ref{P} and \ref{Pep} individually.
 \begin{lem}[Br\'ezis {\cite[Theoreme 3.6]{Brezis}}]\label{pre3}
 Let $X$ be a Hilbert space and 
 let $\psi:X \to \overline{\mathbb{R}}$ 
 be a proper l.s.c.\ convex function. 
 If $u_{0}\in D(\psi)$ and $\tilde{f}\in L^2(0, T; X)$, 
 then 
 there exists a unique function $u$ 
 such that $u\in H^1(0, T; X)$, 
 $u(t)\in D(\partial\psi)$ for a.a.\ $t\in(0, T)$ and 
 $u$ solves the following initial value problem:
    \begin{equation*}
       \begin{cases}
         u'(t) + \partial\psi(u(t)) \ni \tilde{f}(t) 
         \quad \mbox{in}\ X \quad \mbox{for a.a.}\ t\in(0, T),
         \\[2mm]
         u(0) = u_0 \quad \mbox{in}\ X. 
       \end{cases}
    \end{equation*}
 \end{lem}

\vspace{10pt}

\section{Existence of solutions to \ref{P}}\label{Sec3}
%
\subsection{Convex function 
for Proof of Theorem \ref{maintheorem1}}

 Let $H$, $V$ and $W$ be as in \eqref{HV} and \eqref{W}.
 We define a function $\phi : V^{*}\to\overline{\mathbb{R}}$                   
 as
 $$
 \phi(z)=
   \begin{cases}
   \displaystyle\int_{\Omega}\hat{\beta}(z(x))\,dx 
   & \mbox{if}\ 
   z\in D(\phi)
   :=\{z\in H\ |\ \hat{\beta}(z)\in L^1(\Omega) \},\ 
   \\[3mm]
   +\infty 
   & \mbox{otherwise}.
   \end{cases}
 $$
 \begin{lem} \label{lem.3.1}
 Let $\phi$ be as above. Then 
 $\phi$ is a proper l.s.c.\ convex function on $V^{*}$.                              
 \end{lem}
 \begin{proof}
 It follows that $\phi$ is proper and convex 
since $0 \in D(\phi)$ and $\hat{\beta}$ 
 is convex. 
 To prove the lower semicontinuity of $\phi$ on $V^*$ 
  let $\{{z}_{n}\}$ be a sequence in 
 $D(\phi)$ 
 such that $z_{n}\to z$ 
 in $V^{*}$ as $n\to+\infty$. 
 We put $\alpha:=\liminf_{n\to+\infty}\phi(z_{n})$. 
 If $\alpha = +\infty$, 
 then $\phi(z) \leq +\infty = \alpha 
 = \liminf_{n\to+\infty}\phi(z_{n})$. 
 We assume that $\alpha < +\infty$. 
 Then there exists a subsequence $\{z_{n_{k}}\}$ of $\{z_{n}\}$ 
 such that $\phi(z_{n_{k}}) \nearrow \alpha$ 
 as $k \to +\infty$ 
 and hence, 
 $\alpha \geq \phi(z_{n_{k}}) 
 = \int_{\Omega}\hat{\beta}(z_{n_{k}}) 
 \geq c_{1}|z_{n_{k}}|_{H}^2$ by (C1). 
 Thus $z_{n_{k}} \rightharpoonup z$ 
 weakly in $H$ as $k \to +\infty$. 
 Now let $\phi_H := \phi|_H$.
 Since $\hat{\beta}$ is proper l.s.c.\ convex, 
 the function $\phi_H$ is also proper l.s.c.\ convex on $H$
 and hence $\phi_H$ is weakly l.s.c.\ on $H$. 
 So it follows that 
 \[
     \phi_H (z) \le \liminf_{k\to+\infty} \phi_H (z_{n_k}) 
                       =  \liminf_{k\to+\infty} \int_{\Omega}\hat{\beta}(z_{n_{k}})
                       \le \alpha < +\infty.
\]
 Thus we see that $z \in D(\phi_{H}) = D(\phi)$ and 
 $\phi(z) = \phi_H (z) 
 \leq \alpha = \liminf_{n\to+\infty}\phi(z_{n})$.
 \end{proof}
 %
 %
 The following lemma plays an important role 
 in our proof (cf.\ \cite[Lemma 4.1]{FKP-2004}).
 \begin{lem} \label{lem.3.2}
 Let 
 $z\in D(\partial\phi)
 :=\{z\in D(\phi)\ |\ \partial\phi(z)\in V^{*}\} 
 \subset D(\phi)$. 
 Then $z^{*} \in \partial\phi(z)$ 
 in $V^{*}$ if and only if 
    \begin{equation}\label{subdiff1}
    F^{-1}z^{*} = \beta(z)
    \end{equation}
 Consequently, $\partial\phi$ 
 is single-valued and 
 for all $z\in D(\partial\phi)$ it holds that 
 $\beta(z)\in V$ and 
    \begin{equation}\label{subdiff2}
    \partial\phi(z) = F\beta(z). 
    \end{equation}
 \end{lem}
 \begin{proof}   
 Let $z\in D(\partial\phi)$ and $z^{*} \in \partial\phi(z)$. 
 Then it follows from the inclusion 
 $D(\partial\phi)\subset D(\phi)$ that $z\in D(\phi)$. 
 Hence we have by the definition of $\partial\phi$, 
    $$
    (z^{*}, w-z)_{V^{*}} 
    \leq 
    \int_{\Omega}
    \bigl(\hat{\beta}(w)-\hat{\beta}(z)\bigr)
    \quad \mbox{for all}\ w\in D(\phi).
    $$
 Here, choose $w = z \pm \lambda\psi$ ($\lambda > 0$)  
 in the above inequality 
 for each $\psi\in C_{\mathrm{c}}^{\infty}(\Omega)$. 
 Noting by (C1) that $z \pm \lambda\psi \in D(\phi)$, we obtain
  \begin{equation}\label{hasamiuti}
   \int_{\Omega}
             \frac{\hat{\beta}(z) - \hat{\beta}(z-\lambda\psi)}{\lambda} 
   \leq (z^{*}, \psi)_{V^{*}} 
   \leq \int_{\Omega}
             \frac{\hat{\beta}(z+\lambda\psi) - \hat{\beta}(z)}{\lambda}.
   \end{equation}
Here, since $\beta = \partial\hat{\beta}$, 
it follows from the definition of subdifferentials and 
the convexity and nonnegativity of $\hat{\beta}$ 
that 
\begin{align*}
\beta(z)\psi 
&\leq \frac{\hat{\beta}(z + \lambda\psi) - \hat{\beta}(z)}{\lambda} 
 = \frac{\hat{\beta}\bigl(\lambda (z+\psi)+(1-\lambda)z\bigr)
    - \hat{\beta}(z)}{\lambda} 
\leq \hat{\beta}(z + \psi), 
\\[2mm]
-\hat{\beta}(z - \psi)
&\leq \frac{\hat{\beta}(z)
    - \hat{\beta}\bigl(\lambda (z-\psi)+(1-\lambda)z\bigr)}{\lambda} 
= \frac{\hat{\beta}(z) - \hat{\beta}(z - \lambda\psi)}{\lambda} 
\leq \beta(z)\psi, 
\end{align*}
and hence we observe 
\begin{align*}
&\left|\frac{\hat{\beta}(z + \lambda\psi) - \hat{\beta}(z)}{\lambda} \right| 
\leq |\beta(z)\psi| + |\hat{\beta}(z + \psi)|, 
\\[2mm]
&\left|\frac{\hat{\beta}(z) - \hat{\beta}(z - \lambda\psi)}{\lambda} \right| 
\leq |\beta(z)\psi| + |\hat{\beta}(z - \psi)|. 
\end{align*}
Noting that $|\beta(z)\psi| + |\hat{\beta}(z \pm \psi)| \in L^1(\Omega)$ and 
$\hat{\beta}$ is differentiable because of (C1) and
passing to the limit $\lambda \searrow 0$ in \eqref{hasamiuti}, 
we infer from Lebesgue's convergence theorem that
    $$
    (z^{*}, \psi)_{V^{*}} 
    = \int_{\Omega}\hat{\beta}\,'(z)\psi
    = \int_{\Omega}\beta(z)\psi
    \quad \mbox{for all}\ 
    \psi \in C_{\mathrm{c}}^{\infty}(\Omega).
    $$
 Writing as  
 $(z^{*}, \psi)_{V^{*}} 
 = \bigl(F^{-1}z^{*}, \psi \bigr)_{H}$ 
 by \eqref{innerVstar}, 
 we see that 
    $$
    \int_{\Omega}\bigl(F^{-1}z^{*}\bigr)\psi
    = \int_{\Omega}\beta(z)\psi 
    \quad \mbox{for all}\ 
    \psi \in C_{\mathrm{c}}^{\infty}(\Omega).
    $$
 Thus, Since $\beta(z) \in L_{{\rm loc}}^{1}(\Omega)$ by (C1), 
 it follows from du Bois Reymond's lemma that
 $$
 F^{-1}z^{*} = \beta(z) \quad \mbox{a.e.\ on}\ \Omega.
 $$
 That is, \eqref{subdiff1} holds. Conversely, if 
\eqref{subdiff1} holds, 
then for all $w\in D(\phi)$, 
$$ 
(z^{*}, w-z)_{V^{*}} 
= \bigl(F^{-1}z^{*}, w-z \bigr)_{H} 
= \int_\Omega \beta(z)(w-z) 
\leq 
\int_{\Omega} \bigl(\hat{\beta}(w)-\hat{\beta}(z)\bigr), 
$$ 
where we have used $\beta = \partial\hat{\beta}$, 
and hence $z^{*} \in \partial\phi(z)$. 
Therefore we conclude that 
$\partial\phi$ 
is single-valued and 
for all $z\in D(\partial\phi)$, $\beta(z)\in V$ 
and \eqref{subdiff2} holds.
\end{proof} 
 Now we prove the first main theorem.
 %
 %
%
\subsection{Proof of Theorem \ref{maintheorem1}}

 \begin{prth1.1} 
 To prove existence of weak solutions to \ref{P} 
 we turn our eyes to the following initial value problem \eqref{K}:
    \begin{equation}\label{K}
       \begin{cases}
         u'(t) + \partial\phi(u(t)) = Ff(t) 
         \quad \mbox{in} \ V^{*}\quad \mbox{for a.e.}\ t\in[0, T],
         \\[2mm]
         u(0) = u_0 \quad \mbox{in}\ V^{*}. 
       \end{cases}
    \end{equation}
 Thanks to Lemma \ref{pre3}, 
 there exists a unique solution $u\in H^1(0, T; V^{*})$ 
 of \eqref{K} such that $u(t)\in D(\partial\phi)$ 
 for a.a.\ $t\in (0, T)$.
 Putting $\mu(t):=-F^{-1}(u'(t))$, 
 we deduce from \eqref{defF}, \eqref{innerVstar} and \eqref{subdiff2} 
 that $\mu\in L^2(0, T; V)$ and 
 $(u, \mu)$ satisfies \eqref{de4}-\eqref{de6}. 
 
 Next we show \eqref{es1}. 
 It follows from the equation in \eqref{K} that 
    \begin{align*}
    |u'(s)|_{V^{*}}^2 
    &= \bigl(u'(s), u'(s)\bigr)_{V^{*}} \\[1mm]
    &= \bigl(u'(s), -\partial\phi(u(s)) + Ff(s)\bigr)_{V^{*}} \\[1mm]
    &= -\bigl(u'(s), \partial\phi(u(s))\bigr)_{V^{*}} + 
                                     \bigl(u'(s), Ff(s)\bigr)_{V^{*}}. 
    \end{align*}
 Here Lemma \ref{pre2} gives  
    $$
    \bigl(u'(s), \partial\phi(u(s))\bigr)_{V^{*}} 
    = \frac{d}{ds}\phi(u(s)), 
    $$
 and \eqref{innerVstar} and Young's inequality yield 
    $$
    (u'(s), Ff(s))_{V^{*}} 
    = \langle u'(s), f(s) \rangle_{V^{*}, V} 
    \leq \frac{1}{2}|u'(s)|_{V^{*}}^2 + \frac{1}{2}|f(s)|_{V}^2.
    $$
 Therefore we obtain
    $$
    \frac{1}{2}|u'(s)|_{V^{*}}^2
    \leq -\frac{d}{ds}\phi(u(s)) + \frac{1}{2}|f(s)|_{V}^2.
    $$
 Integrating this inequality yields
    $$
    \frac{1}{2}\int_{0}^{t}|u'(s)|_{V^{*}}^2\,ds 
    \leq -\phi(u(t)) + \phi(u_{0}) + 
                           \frac{1}{2}|f|_{L^2(0, T; V)}^2,
    $$
 i.e., 
    $$
    \frac{1}{2}\int_{0}^{t}|u'(s)|_{V^{*}}^2\,ds + 
    \int_{\Omega}\hat{\beta}(u(t)) 
    \leq 
    \int_{\Omega}\hat{\beta}(u_{0}) + 
    \frac{1}{2}|f|_{L^2(0, T; V)}^2.
    $$
 Since (C1) implies 
    $$
    \int_{\Omega}\hat{\beta}(u(t)) \geq c_{1}|u(t)|_{H}^2, 
    $$
 we see that 
    $$
    \int_{0}^{t}|u'(s)|_{V^{*}}^2\,ds + 2c_{1}|u(t)|_{H}^2
    \leq 
    2\int_{\Omega}\hat{\beta}(u_{0}) + |f|_{L^2(0, T; V)}^2 =: M_1.
    $$
 This implies \eqref{es1}. Moreover, \eqref{es1} shows that $u\in L^{\infty}(0, T; H)$. 
 
 Next we show \eqref{es2}. 
 Since $\mu(s) = -F^{-1}\bigl(u'(s)\bigr)$, we have from 
 \eqref{defF} and \eqref{innerVstar} that 
 \begin{align*}
 \int_{0}^{t}|\mu(s)|_{V}^2\,ds 
 = \int_{0}^{t}\bigl|F^{-1}\bigl(u'(s)\bigr)\bigr|_{V}^2\,ds 
 = \int_{0}^{t}\bigl|u'(s)\bigr|_{V^{*}}^2\,ds. 
 \end{align*}
 Thus we obtain \eqref{es2} from \eqref{es1}.
 
 Next we verify \eqref{es3}. 
 From \eqref{de5} and Young's inequality we infer 
    \begin{align*}
    |\beta(u(s))|_{V}^2 &= \bigl(\beta(u(s)), \beta(u(s))\bigr)_{V} \\[1mm]
    &= \bigl(\mu(s) + f(s), \beta(u(s))\bigr)_{V} \\[1mm]
    &\leq |\mu(s)|_{V}^2 + |f(s)|_{V}^2 + 
                        \frac{1}{2}|\beta(u(s))|_{V}^2.
    \end{align*}
 Therefore, 
    $$
    \int_{0}^{t}|\beta(u(s))|_{V}^2\,ds 
    \leq 2\int_{0}^{t}|\mu(s)|_{V}^2\,ds + 
    2|f|_{L^2(0, T; V)}^2.
    $$
 Consequently, \eqref{es3} holds from \eqref{es2}. 
 \qed
 \end{prth1.1}

 \vspace{10pt}
 
 \section{Existence of solutions to \ref{Pep}}\label{Sec4}
%
\subsection{Preliminaries for \ref{Pep}}

 We first give a useful inequality.
 \begin{lem}\label{pre5}
 Let $\beta$ be a single-valued maximal monotone function 
 as in Section {\rm \ref{Sec1}}. 
 Then
    \begin{align*}
    &\bigl(-\Delta u, \beta_{\lambda}(u)\bigr)_{H} \geq 0 
    \quad \mbox{for all}\ 
    u\in W, 
    \\[1mm]
    &\bigl(-\Delta u, \beta(u)\bigr)_{H} \geq 0 
    \quad \mbox{for all}\ 
    u\in W\ \mbox{with}\ \beta(u)\in H,
    \end{align*}
 where $W=\bigl\{z\in H^2(\Omega)\ |\ \partial_{\nu}z = 0 \quad 
    \mbox{a.e.\ on}\ \partial\Omega\bigr\}$ and 
 $\{\beta_{\lambda}\}_{\lambda>0}$ 
 is the Yosida approximation 
 of 
 $\beta$: $\beta_{\lambda}
 :=\lambda^{-1}\bigl(I-(I+\lambda\beta)^{-1}\bigr)$.
 \end{lem}
 \begin{proof}
 It follows from 
 Okazawa \cite[Proof of Theorem 3 with $a=b=0$]{O-1983} 
 that 
 $$
 \bigl(-\Delta u, \beta_{\lambda}(u)\bigr)_{H} \geq 0 
    \quad \mbox{for all}\ 
    u\in W\ \mbox{and}\ \lambda>0.
 $$
 Noting that $\beta_{\lambda}(u) \to \beta(u)$ in $H$ 
 as $\lambda \searrow 0$ if $\beta(u)\in H$ 
 (see e.g., {\cite[Proposition 2.6]{Brezis}} 
 or \cite[Theorem IV.1.1]{S-1997}), 
 we can obtain the second inequality. 
 \end{proof}

 The above and the next lemmas will be used 
 in order to regard \ref{Pep} as a problem  
 of the form stated in Lemma \ref{pre3}.
 \begin{lem}\label{pre6}
 Let $A$ and $B$ be maximal monotone operators in $H$ 
 such that 
    \begin{enumerate}
    \item[{\rm (i)}] $D(A) \cap D(B) \neq \emptyset$, 
    \item[{\rm (ii)}] $(Av,\ B_{\lambda}v)_{H} \geq 0$ 
    \quad for all $v \in D(A)$ 
    and $\lambda>0$, 
    \end{enumerate}
 where $\{B_{\lambda}\}_{\lambda>0}$ is 
 the Yosida approximation of $B$. 
 Then $A + B$ is maximal monotone.
 \end{lem}
 \begin{proof}
 We can show this lemma 
 by applying Barbu {\cite[Theoreme II.3.6]{Barbu2}}.
 \end{proof}

%
\subsection{Convex function for Proof of Theorem \ref{maintheorem2}}

 Let $\ep > 0$. Then we define a function 
 $\phi_{\ep} : V^{*}\to\overline{\mathbb{R}}$                     
 as
 $$
 \phi_{\ep}(z)=
   \begin{cases}
   \displaystyle\frac{\ep}{2}\int_{\Omega}
                   \bigl(|z(x)|^2 + |\nabla z(x)|^2\bigr)\,dx + 
   \int_{\Omega}\hat{\beta}(z(x))\,dx + 
   \int_{\Omega}\hat{\pi_{\ep}}(z(x))\,dx &
   \\
   \hspace{50mm}\mbox{if}\ 
   z\in D(\phi_{\ep}):=\{z\in V\ |\ 
                   \hat{\beta}(z)\in L^1(\Omega) \},\ &
   \\[3mm]
   +\infty 
   \hspace{42.7mm}\mbox{otherwise}. &
   \end{cases}
 $$
 \begin{lem}\label{lem.4.1}
 Let $\phi_{\ep}$ be as above. 
 Then there exists $\overline{\ep} \in (0, 1]$ 
 such that for all $\ep \in (0, \overline{\ep}]$, 
 $\phi_{\ep}$ is a proper 
 l.s.c.\ convex function on $V^{*}$.                            
 \end{lem}
 \begin{proof}
 Since 
 $0$ belongs to $D(\phi_{\ep})$ and 
 $r \mapsto \hat{\beta}(r)$, 
 $r\mapsto\frac{\ep}{2}r^2+\hat{\pi_{\ep}}(r)$ are convex, 
  it follows that $\phi_{\ep}$ is proper and convex.
 To prove the lower semicontinuity of $\phi_{\ep}$ in $V^{*}$ 
 let $\{{z}_{n}\}$ be a sequence in 
 $D(\phi_{\ep})$ 
 such that $z_{n}\to z$ 
 in $V^{*}$ as $n\to+\infty$. 
 We put $\alpha:=\liminf_{n\to+\infty}\phi_{\ep}(z_{n})$. 
 If $\alpha=+\infty$, then $\phi_{\ep}(z) \leq +\infty 
 = \alpha=\liminf_{n\to+\infty}\phi_{\ep}(z_{n})$. 
 We assume that $\alpha < +\infty$. 
 Then there exists a subsequence $\{z_{n_{k}}\}$ of $\{z_{n}\}$ 
 such that $\phi_{\ep}(z_{n_{k}}) \nearrow \alpha$ 
 as $k\to+\infty$ and hence, 
 \begin{align*}
 \alpha 
 \geq \phi_{\ep}(z_{n_{k}}) 
   = \int_{\Omega}\hat{\beta}(z_{n_{k}}) 
     + \int_{\Omega}\hat{\pi_{\ep}}(z_{n_{k}}) 
     + \frac{\ep}{2}|z_{n_{k}}|_{V}^2
 \geq c_{1}\int_{\Omega}|z_{n_{k}}|^2  
       + \int_{\Omega}\hat{\pi_{\ep}}(z_{n_{k}})
       + \frac{\ep}{2}|z_{n_{k}}|_{V}^2.
 \end{align*}
 Here, we deduce from (C4) that 
 there exists $\overline{\ep}\in(0, 1]$ such that 
 $c_2(\ep) < 2c_{1}$ 
 for all $\ep \in (0, \overline{\ep}]$.
 The definition of $\hat{\pi_{\ep}}$ shows that for all $\ep \in (0, \overline{\ep}]$,
 $$
 |\hat{\pi_{\ep}}(r)| 
 = \left|\int_{0}^{r}(\pi_{\ep}(s) - \pi_{\ep}(0))\,ds\right| 
 \leq |\pi_{\ep}'|_{L^{\infty}(\mathbb{R})}
                 \left|\int_{0}^{r}|s|\,ds\right| 
 \leq \frac{1}{2}c_2(\ep)|r|^2 
 \leq c_{1}|r|^2.
 $$ 
 Hence $\alpha \geq \frac{\ep}{2}|z_{n_{k}}|_{V}^2$. 
 Thus $z_{n_{k}} \rightharpoonup z$ 
 weakly in $V$ as $k \to +\infty$. 
 Now let $\phi_{\ep, V} := \phi_{\ep}|_{V}$. 
 Since $\hat{\beta}$ is proper l.s.c.\ convex and 
 $r\mapsto\frac{\ep}{2}r^2+\hat{\pi_{\ep}}(r)$ is convex, 
 the function $\phi_{\ep, V}$ is also proper l.s.c.\ convex on $V$ and hence 
 $\phi_{\ep, V}$ is weakly l.s.c.\ on $V$. So it follows that
 \begin{equation*}
 \phi_{\ep, V}(z) \leq \liminf_{k\to+\infty} \phi_{\ep, V}(z_{n_{k}}) 
 \leq \alpha < +\infty. 
 \end{equation*}
 Consequently, $z \in D(\phi_{\ep, V}) = D(\phi_{\ep})$ and 
 $\phi_{\ep}(z) 
 = \phi_{\ep, V}(z)
 \leq \alpha = \liminf_{n\to+\infty}\phi_{\ep}(z_{n})$.
 \end{proof}
 %
 %
 \begin{lem}\label{lem.4.2}
 Define a function 
 $\phi_{\ep}^{H} : H \to \overline{\mathbb{R}}$ as 
 $$
 \phi_{\ep}^{H}(w)=
   \begin{cases}
   \dfrac{1}{2}\displaystyle\int_{\Omega}|w|^2 + 
   \dfrac{\ep}{2}\displaystyle\int_{\Omega}
                   \bigl(|w|^2 + |\nabla w|^2\bigr) + 
   \int_{\Omega}\hat{\beta}(w) + 
   \int_{\Omega}\hat{\pi_{\ep}}(w)&
   \\
   \hspace{50mm}\mbox{if}\ 
   w\in D(\phi_{\ep}^{H}):=\{w\in V\ |\ 
                     \hat{\beta}(w)\in L^1(\Omega) \},\ &
   \\[3mm]
   +\infty 
   \hspace{42.7mm}\mbox{otherwise}. &
   \end{cases}
 $$
 Then $\phi_\ep^H$ is a proper l.s.c.\ convex function on $H$ and 
 \begin{equation}\label{inW}
 D(\partial\phi_{\ep}^{H}) \subset W.
 \end{equation}
 \end{lem}
 \begin{proof}
 As in the proof of Lemma \ref{lem.4.1}, we first observe that 
 $\phi_\ep^H$ is a proper l.s.c.\ convex function on $H$. 
 Next we set $\phi_{H}^{(1)} : H \to \overline{\mathbb{R}}$ as 
 $$
 \phi_{H}^{(1)}(w)=
   \begin{cases}
   \dfrac{1}{2}\displaystyle\int_{\Omega}|w|^2 + 
   \dfrac{\ep}{2}\displaystyle\int_{\Omega}
                   \bigl(|w|^2 + |\nabla w|^2\bigr) + 
   \int_{\Omega}\hat{\pi_{\ep}}(w)
   &\mbox{if}\ w\in D(\phi_{H}^{(1)}):=V,\ 
   \\[3mm]
   +\infty &\mbox{otherwise} 
   \end{cases}
 $$
 and $\phi_{H}^{(2)} : H \to \overline{\mathbb{R}}$ as 
 $$
 \phi_{H}^{(2)}(w)=
   \begin{cases}
   \displaystyle\int_{\Omega}\hat{\beta}(w)
   &\mbox{if}\ 
   w\in D(\phi_{H}^{(2)}):=\{w\in H\ |\ 
                     \hat{\beta}(w)\in L^1(\Omega) \},
   \\[3mm]
   +\infty &\mbox{otherwise}.
   \end{cases}
 $$
 Then $\phi_{H}^{(1)}$ and $\phi_{H}^{(2)}$ are proper 
 l.s.c.\ convex functions on $H$ and 
\begin{align}
&w \in D(\partial\phi_{H}^{(1)}) \quad \Longrightarrow \quad
w \in W\ \mbox{and}\ 
\partial\phi_{H}^{(1)}(w) = w + \ep(-\Delta+I)w + \pi_{\ep}(w), 
\label{sub1} 
\\ 
&w \in D(\partial\phi_{H}^{(2)}) \quad \Longrightarrow \quad \partial\phi_{H}^{(2)}(w) = \beta(w). 
\label{sub2}
\end{align} 
Since \eqref{sub2} is well-known 
(see e.g., \cite[Example 2.8.3]{Brezis}, \cite[Example II.8.B]{S-1997}), 
we verify only \eqref{sub1}.
 Let $w\in D(\partial\phi_{H}^{(1)})$ and 
 $w^{*} \in \partial\phi_{H}^{(1)}(w)$. 
 Then it follows from the inclusion 
 $D(\partial\phi_{H}^{(1)})\subset D(\phi_{H}^{(1)})$ 
 that $w\in D(\phi_{H}^{(1)})$. 
 Hence we have from the definition of $\partial\phi_{H}^{(1)}$ that 
    \begin{align*}
    (w^{*}, \tilde{w}-w)_{H} 
    &\leq 
    \frac{1}{2}
    \int_{\Omega}\bigl(|\tilde{w}|^2 - |w|^2\bigr) +
    \frac{\ep}{2}
    \int_{\Omega}\bigl(|\tilde{w}|^2 - |w|^2\bigr) + 
      \frac{\ep}{2}
      \int_{\Omega}
      \bigl(|\nabla \tilde{w}|^2 - |\nabla w|^2\bigr)
    \\ 
    &\hspace{4.5mm}
     + \int_{\Omega}
          \bigl(\hat{\pi_{\ep}}(\tilde{w})
                            -\hat{\pi_{\ep}}(w)\bigr).
    \end{align*}
 Here, choose $\tilde{w} = w \pm \lambda v$ ($\lambda > 0$) 
 in the above inequality 
 for each $v \in V$ and divide 
 the both sides by $\lambda$ and finally pass to 
 the limit $\lambda \searrow 0$. Then we obtain 
 $$
    (w^{*}, v)_{H} 
    = \int_{\Omega}wv 
      + \ep\left(\int_{\Omega}wv + 
          \int_{\Omega}\nabla w\cdot\nabla v \right) 
      + \int_{\Omega}\pi_{\ep}(w)v
    \quad \mbox{for all}\ v \in V.
 $$
 Hence we see that
 $$
 \int_{\Omega}wv + \int_{\Omega}\nabla w\cdot\nabla v 
 = 
 \int_{\Omega}\frac{w^{*}-w-\pi_{\ep}(w)}{\ep}v 
 \quad \mbox{for all}\ v \in V.
 $$
 Thus we derive that $w \in W$ and 
 $$
 w^{*} = w + \ep(-\Delta+I)w + \pi_{\ep}(w).
 $$
 That is, \eqref{sub1} holds. 
 Now we show that \eqref{inW}.
 If 
 \begin{equation}\label{mono}
 \bigl(\partial\phi_{H}^{(1)}(w),\ 
                          \beta_{\lambda}(w) \bigr)_{H} 
 \geq 0 
 \quad \mbox{for all}\ 
 w \in D(\partial\phi_{H}^{(1)}),
 \end{equation}
 then $\partial\phi_{H}^{(1)} + \partial\phi_{H}^{(2)}$ is 
 maximal monotone by Lemma \ref{pre6} 
 and hence we have 
 $$
 \partial\phi_{\ep}^{H} 
 = \partial(\phi_{H}^{(1)} + \phi_{H}^{(2)}) 
 = \partial\phi_{H}^{(1)} + \partial\phi_{H}^{(2)}, 
 $$
 with
 \begin{equation}\label{cap}
 D(\partial\phi_{\ep}) 
 = D(\partial\phi_{H}^{(1)}) \cap D(\partial\phi_{H}^{(2)}); 
 \end{equation}
 note that $\phi_H=\phi_H^{(1)}+\phi_H^{(2)}$ is also proper l.s.c.\ convex and $\partial(\phi_H^{(1)}+\phi_H^{(2)}) \subset \partial\phi_H^{(1)} + \partial\phi_H^{(2)}$.
 We can show \eqref{mono} by using Lemma \ref{pre5}. Indeed, 
 since the function $r \mapsto \frac{\ep}{2}r^2 + \hat{\pi_{\ep}}(r)$ 
 is convex, it follows that $r \mapsto \ep r + \pi_{\ep}(r)$ is monotone,  
 so that the monotonicity of $\beta_{\lambda}$ yields 
 $$
 (\ep w + \pi_{\ep}(w),\ \beta_{\lambda}(w))_{H} \geq 0, 
 $$ 
 and hence we see from Lemma \ref{pre5} that 
 \begin{align*}
 \bigl(\partial\phi_{H}^{(1)}(w),\ 
                          \beta_{\lambda}(w) \bigr)_{H} 
 &= (w + \ep(-\Delta+I)w 
               + \pi_{\ep}(w),\ \beta_{\lambda}(w))_{H} 
 \\
 &= (w,\ \beta_{\lambda}(w))_{H} 
    + \ep(-\Delta w,\ \beta_{\lambda}(w))_{H} 
    + (\ep w + \pi_{\ep}(w),\ \beta_{\lambda}(w))_{H} 
 \\
 &\geq 0. 
 \end{align*}
 Therefore we obtain \eqref{cap}. On the other hand, 
 we infer from \eqref{sub1} that 
 \begin{align}\label{Dsub1}
 D(\partial\phi_{H}^{(1)}) 
 &= \bigl\{w \in D(\phi_{H}^{(1)})\ \bigl|\ 
                       \partial\phi_{H}^{(1)}(w) \in H \bigr\} 
 \\ \notag
 &= \bigl\{w \in V\ \bigl|\ w \in W,\ 
             w + \ep(-\Delta+I)w + \pi_{\ep}(w) \in H \bigr\} 
 \\ \notag
 &= W
 \end{align}
 and from \eqref{sub2} that 
 \begin{align}\label{Dsub2}
 D(\partial\phi_{H}^{(2)}) 
 &= \bigl\{w \in D(\phi_{H}^{(2)})\ \bigl|\ 
                       \partial\phi_{H}^{(2)}(w) \in H \bigr\} 
 \\ \notag
 &= \{w \in H\ \bigl|\ \beta(w) \in H \}. 
 \end{align}
 Thus, connecting \eqref{Dsub1} and \eqref{Dsub2} to \eqref{cap} 
 gives \eqref{inW}.
 \end{proof}
 \begin{lem}\label{lem.4.3}
 Let $z\in D(\partial\phi_{\ep}) 
 := \{z\in D(\phi_{\ep})\ |\ \partial\phi_{\ep}(z)\in V^{*}\} 
 \subset D(\phi_{\ep})$. 
 Then $z^{*} \in \partial\phi_{\ep}(z)$ 
 in $V^{*}$ if and only if $z\in W$ and 
    \begin{equation}\label{epsubdiff1}
    F^{-1}z^{*} 
    = 
    \ep(-\Delta + I)z + \beta(z) + \pi_{\ep}(z) 
    \end{equation}
 Consequently, $\partial\phi_{\ep}$ 
 is single-valued and 
 for all $z\in D(\partial\phi_{\ep})$ 
 it holds that
 \begin{align}
 &z\in W, \quad 
 \ep(-\Delta + I)z + \beta(z) + \pi_{\ep}(z)\in V \quad \notag
 \mbox{and} 
 \\[1mm]
 \label{epsubdiff2}
 &\partial\phi_{\ep}(z) 
 = F(\ep(-\Delta + I)z + \beta(z) + \pi_{\ep}(z)).
 \end{align}
 \end{lem}
 \begin{proof}   
 Let $z\in D(\partial\phi_{\ep})$ and 
 $z^{*} \in \partial\phi_{\ep}(z)$. Noting that 
 $D(\partial\phi_{\ep})\subset D(\phi_{\ep})$, 
 we see from the definition of $\partial\phi_{\ep}$ 
 that for all $w \in D(\partial\phi_{\ep})$, 
    \begin{align*}
    (z^{*}, w-z)_{V^{*}} 
    &\leq 
    \frac{\ep}{2}
    \int_{\Omega}\bigl(|w|^2 - |z|^2\bigr) +
    \frac{\ep}{2}
    \int_{\Omega}
          \bigl(|\nabla w|^2 - |\nabla z|^2\bigr) 
    \\ &\quad\ 
    + \int_{\Omega}
    \bigl(\hat{\beta}(w)-\hat{\beta}(z)\bigr) +
    \int_{\Omega}
          \bigl(\hat{\pi_{\ep}}(w)
                            -\hat{\pi_{\ep}}(z)\bigr).
    \end{align*}
 Here, choose $w = z \pm \lambda\psi$ $\lambda > 0$  
 in the above inequality 
 for each $\psi \in {\cal D}(\Omega) := C_{\mathrm{c}}^{\infty}(\Omega)$ and divide 
 the both sides by $\lambda$ and finally pass to 
 the limit $\lambda \searrow 0$. Then 
 for all $\psi \in {\cal D}(\Omega)$, 
 we obtain  
    \begin{align*}
      (z^{*}, \psi)_{V^{*}} 
    &= \ep\left(\int_{\Omega}z\psi 
    + \int_{\Omega}
          \nabla z\cdot\nabla\psi\right) 
    + \int_{\Omega}\beta(z)\psi 
                + \int_{\Omega}
                       \pi_{\ep}(z)\psi. 
    \end{align*} 
 The relation 
 $(z^{*},\ \psi)_{V^{*}} 
                      = \bigl(F^{-1}z^{*},\ \psi)_{H}$ 
 and the arbitrariness of 
 $\psi \in {\cal D}(\Omega)$ 
 yield 
    \begin{align*}
    \int_{\Omega} z(-\Delta + I)\psi
    = 
    \int_{\Omega}
    \frac{F^{-1}z^{*} - \beta(z) - \pi_{\ep}(z)}{\ep}\psi 
    \quad 
    \mbox{for all}\ \psi \in {\cal D}(\Omega).
    \end{align*}
 This implies that 
    \begin{align*}
    (-\Delta+I)_{{\cal D}'(\Omega)}z 
    = 
    \dfrac{F^{-1}z^{*} - \beta(z) - \pi_{\ep}(z)}{\ep} 
    \quad \mbox{in}\ {\cal D}' (\Omega), 
    \end{align*}
 where ${\cal D}' (\Omega)$ is the space of distributions on $\Omega$. 
 Thus we see that 
 \begin{equation*}
 \partial\phi_{\ep}(z) 
 = F\bigl(\ep(-\Delta+I)_{{\cal D}'(\Omega)}z 
          + \beta(z) + \pi_{\ep}(z)\bigr) 
 \quad \mbox{for all}\ z \in D(\partial\phi_{\ep}).
 \end{equation*}
 It suffices from Lemma \ref{lem.4.2} to prove the following 
 inclusion relation: 
 \begin{equation}\label{inclu}
 D(\partial\phi_{\ep}) \subset D(\partial\phi_{\ep}^{H}).
 \end{equation}
 It holds that 
 \begin{align*}
 D(\partial\phi_{\ep}) 
 &= \{w \in D(\phi_{\ep})\ |\ 
                \partial\phi_{\ep}(w) \in V^{*} \} 
 \\[1mm]
 &= \Bigl\{w \in V\ \Bigl|\ 
      \hat{\beta}(w) \in L^1(\Omega),\ 
          F\bigl(\ep(-\Delta+I)_{{\cal D}'(\Omega)}w 
          + \beta(w) + \pi_{\ep}(w)\bigr)\in V^{*} \Bigr\} 
 \\[1mm]
 &= \Bigl\{w \in V\ \Bigl|\ 
      \hat{\beta}(w) \in L^1(\Omega),\ 
          \ep(-\Delta+I)_{{\cal D}'(\Omega)}w 
          + \beta(w) + \pi_{\ep}(w) \in V \Bigr\}
 \end{align*}
 and it follows that 
 \begin{align*}
 D(\partial\phi_{\ep}^{H}) 
 &= \{w \in D(\phi_{\ep}^{H})\ |\ 
                \partial\phi_{\ep}^{H}(w) \in H \} 
 \\[1mm]
 &= \Bigl\{w \in V\ \Bigl|\ 
      \hat{\beta}(w) \in L^1(\Omega),\ 
          w + \ep(-\Delta+I)_{{\cal D}'(\Omega)}w 
          + \beta(w) + \pi_{\ep}(w) \in H \Bigr\}
 \\[1mm]
 &= \Bigl\{w \in V\ \Bigl|\ 
      \hat{\beta}(w) \in L^1(\Omega),\ 
          \ep(-\Delta+I)_{{\cal D}'(\Omega)}w 
          + \beta(w) + \pi_{\ep}(w) \in H \Bigr\}. 
 \end{align*}
 That is, \eqref{inclu} holds.
 \end{proof}
 %
%
\subsection{Proof of Theorem \ref{maintheorem2}}
 
We are now in a position to complete 
the proof of Theorem \ref{maintheorem2}.
 \begin{prth1.2}
 To show existence of weak solutions to \ref{Pep} 
 we consider 
    \begin{equation}\label{Kep}
       \begin{cases}
         u_{\ep}'(t) + \partial\phi_{\ep}(u_{\ep}(t)) = Ff(t) 
         \quad \mbox{in} \ V^{*}
         \quad \mbox{for a.a.}\ t\in[0, T],
         \\[2mm]
         u_{\ep}(0) = u_{0\ep} \quad \mbox{in}\ V^{*}. 
       \end{cases}
    \end{equation}
 In light of Lemma \ref{pre3}, 
 there exists a unique solution $u_{\ep}\in H^1(0, T; V^{*})$ 
 of \eqref{Kep} such that $u_{\ep}(t)\in D(\partial\phi_{\ep})$ 
 for a.a.\ $t\in (0, T)$.
 Putting $\mu_{\ep}(t):=-F^{-1}\bigl(u_{\ep}'(t)\bigr)$, 
 we deduce from \eqref{defF}, \eqref{innerVstar} and \eqref{epsubdiff2} 
 that $\mu_{\ep}\in L^2(0, T; V)$ and 
 $(u_{\ep}, \mu_{\ep})$ satisfies \eqref{de7}-\eqref{de9}. 
 
 Next we show \eqref{epes1}. 
 It follows from the equation in \eqref{Kep} that 
    \begin{align*}
    |u_{\ep}'(s)|_{V^{*}}^2 
    &= \bigl(u_{\ep}'(s), u_{\ep}'(s)\bigr)_{V^{*}} \\
    &= \bigl(u_{\ep}'(s), -\partial\phi_{\ep}(u_{\ep}(s)) + Ff(s)\bigr)_{V^{*}} 
    \\
    &= -\bigl(u_{\ep}'(s), \partial\phi_{\ep}(u_{\ep}(s))\bigr)_{V^{*}} + 
                                     (u_{\ep}'(s), Ff(s))_{V^{*}}. 
    \end{align*}
 Here, we have by Lemma \ref{pre2}, 
    $$
    \bigl(u_{\ep}'(s), \partial\phi_{\ep}(u_{\ep}(s))\bigr)_{V^{*}} 
    = \frac{d}{ds}\phi_{\ep}(u_{\ep}(s)), 
    $$
 and \eqref{innerVstar} and by Young's inequality yield 
    $$
    (u_{\ep}'(s), Ff(s))_{V^{*}} 
    = \langle u_{\ep}'(s), f(s) \rangle_{V^{*}, V} 
    \leq \frac{1}{2}|u_{\ep}'(s)|_{V^{*}}^2 + \frac{1}{2}|f(s)|_{V}^2.
    $$
 Therefore we obtain 
    $$
    \frac{1}{2}|u_{\ep}'(s)|_{V^{*}}^2 
    \leq -\frac{d}{ds}\phi_{\ep}(u_{\ep}(s)) + \frac{1}{2}|f(s)|_{V}^2.
    $$
 Integrating this inequality yields
    $$
    \frac{1}{2}\int_{0}^{t}|u_{\ep}'(s)|_{V^{*}}^2\,ds 
    \leq -\phi_{\ep}(u_{\ep}(t)) + \phi_{\ep}(u_{0\ep}) + 
                           \frac{1}{2}|f|_{L^2(0, T; V)}^2,
    $$
 i.e., 
    \begin{align*}
    &\frac{1}{2}\int_{0}^{t}|u_{\ep}'(s)|_{V^{*}}^2\,ds + 
    \frac{\ep}{2}|u_{\ep}(t)|_{V}^2 + 
    \int_{\Omega}\hat{\beta}(u_{\ep}(t)) + \int_{\Omega}\hat{\pi_{\ep}}(u_{\ep}(t)) 
    \\
    &\leq 
    \frac{\ep}{2}|u_{0\ep}|_{V}^2 + 
    \int_{\Omega}\hat{\beta}(u_{0\ep}) + \int_{\Omega}\hat{\pi_{\ep}}(u_{0\ep}) 
    + \frac{1}{2}|f|_{L^2(0, T; V)}^2.
    \end{align*}
 Here (C1) implies 
    \begin{equation}\label{KY2_y1}
    \int_{\Omega}\hat{\beta}(u_{\ep}(t)) \geq c_{1}|u_{\ep}(t)|_{H}^2. 
    \end{equation}
 Recalling \eqref{a3ineq}, we infer that 
     \begin{equation}\label{KY2_y2}
    |\hat{\pi_{\ep}}(r)| 
    \leq 
    \frac{1}{2}|\pi_{\ep}'|_{L^{\infty}(\mathbb{R})}|r|^2
    \leq 
    \frac{1}{2}c_2(\ep)|r|^2
    \end{equation}
 for all $r\in\mathbb{R}$. Now, from (C4) we deduce that 
 there exists $\overline{\ep}\in(0, 1]$ such that 
 $c_2(\ep) <2c_1$ 
 for all $\ep\in(0, \overline{\ep}]$. Thus combining \eqref{KY2_y1} and  
 \eqref{KY2_y2} gives
    $$
    \int_{\Omega}\hat{\beta}(u_{\ep}(t)) 
    +\int_{\Omega}\hat{\pi_{\ep}}(u_{\ep}(t)) 
    \geq 
    \frac{1}{2}\bigl(2c_{1}-c_2(\ep)\bigr)|u_{\ep}(t)|_{H}^2
    $$
 for a.a. $t\in(0, T)$. Moreover, using \eqref{a4ineq} 
 of (C5) leads to 
    \begin{align*}
    \frac{\ep}{2}|u_{0\ep}|_{V}^2 + 
    \int_{\Omega}\hat{\beta}(u_{0\ep}) + 
    \int_{\Omega}\hat{\pi_{\ep}}(u_{0\ep}) 
    &\leq 
    \frac{c_{3}(\ep)}{2} + c_{3}(\ep) + 
    \frac{1}{2}c_2(\ep)|u_{0\ep}|_{H}^2
    \\
    &\leq \frac{3}{2}c_3(\ep) + \frac{1}{2}c_2(\ep)c_3(\ep).
    \end{align*}
 Therefore we see that 
    \begin{align*}
    \int_{0}^{t}|u_{\ep}'(s)|_{V^{*}}^2\,ds 
    + \ep|u_{\ep}(t)|_{V}^2 + 
    \bigl(2c_{1}-c_2(\ep)\bigr)|u_{\ep}(t)|_{H}^2
    \leq 
    3c_3(\ep) + c_2(\ep)c_3(\ep) + |f|_{L^2(0, T; V)}^2.
    \end{align*}
 This implies \eqref{epes1} 
 with $M_2(\ep) := 3c_3(\ep) + c_2(\ep)c_3(\ep) + |f|_{L^2(0, T; V)}^2$.

 Next we prove \eqref{epes2}.
 Since $\mu_{\ep}(s) = -F^{-1}\bigl(u_{\ep}'(s)\bigr)$, 
 it follows that 
 \begin{align*}
 \int_{0}^{t}|\mu_{\ep}(s)|_{V}^2\,ds 
 = \int_{0}^{t}
           \bigl|-F^{-1}\bigl(u_{\ep}'(s)\bigr)\bigr|_{V}^2\,ds 
 = \int_{0}^{t}\bigl|u_{\ep}'(s)\bigr|_{V^{*}}^2\,ds. 
 \end{align*}
 Therefore we arrive at \eqref{epes2} via \eqref{epes1}.
 
 Next we show \eqref{epes3}. 
 Noting 
 by Lemma \ref{lem.4.3} that $u_{\ep}(s)\in W$ 
 for a.a.\ $s\in (0, T)$ and 
 recalling the definition of $\mu(\cdot)$, 
 the monotonicity of $\beta$ and Lemma \ref{pre5}, we have
    \begin{align*}
     |\beta(u_{\ep}(s))|_{H}^2 
    &= \bigl(\beta(u_{\ep}(s)), \beta(u_{\ep}(s))\bigr)_{H} 
    \\[1mm]
    &= \bigl(\mu_{\ep}(s) - \ep(-\Delta + I)u_{\ep}(s) -\pi_{\ep}(u_{\ep}(s)) + f(s), 
    \beta(u_{\ep}(s))\bigr)_{H} 
    \\[1mm]
    &= \bigl(\mu_{\ep}(s) - \pi_{\ep}(u_{\ep}(s)) + f(s), \beta(u_{\ep}(s))\bigr)_{H} 
    \\
    &\quad\ - \ep\bigl(-\Delta u_{\ep}(s), \beta(u_{\ep}(s))\bigr)_{H}
     - \ep\bigl(u_{\ep}(s), \beta(u_{\ep}(s))\bigr)_{H}
    \\[1mm]
    &\leq(|\mu_{\ep}(s)|_{V} + |\pi_{\ep}(u_{\ep}(s))|_{H} + |f(s)|_{V})|\beta(u_{\ep}(s))|_{H},  
    \end{align*}
 where Young's inequality and (C4) yield 
    \begin{align*}
    &\bigl(|\mu_{\ep}(s)|_{V} + |\pi_{\ep}(u_{\ep}(s))|_{H} + |f(s)|_{V}\bigr)|
                                                                               \beta(u_{\ep}(s))|_{H} 
    \\
    &\leq 
    \frac{1}{2}(|\mu_{\ep}(s)|_{V} + 
    |\pi_{\ep}'|_{L^{\infty}(\mathbb{R})}|u_{\ep}(s)|_{H} + |f(s)|_{V})^2 + 
    \frac{1}{2}|\beta(u_{\ep}(s))|_{H}^2 \\
    &\leq \frac{3}{2}\bigl(|\mu_{\ep}(s)|_{V}^2 + c_2(\ep)^2|u_{\ep}(s)|_{H}^2 
             + |f(s)|_{V}^2\bigr) + \frac{1}{2}|\beta(u_{\ep}(s))|_{H}^2.
    \end{align*}
 Therefore,  
    $$
    \int_{0}^{t}|\beta(u_{\ep}(s))|_{H}^2\,ds 
    \leq 3\Bigl(\int_{0}^{t}|\mu_{\ep}(s)|_{V}^2\,ds + 
    c_2(\ep)^2\int_{0}^{t}|u_{\ep}(s)|_{H}^2\,ds + 
    |f|_{L^2(0, T; V)}^2\Bigr).
    $$
 Thus we obtain \eqref{epes3} by virtue of \eqref{epes1} and \eqref{epes2}. 

 Next we show \eqref{epes4}. 
 It follows from \eqref{de8} that 
    \begin{align*}
    &\int_{0}^{t}|\ep(-\Delta + I)u_{\ep}(s)|_{H}^2\,ds \\
    &= \int_{0}^{t}|\mu_{\ep}(s) - \beta(u_{\ep}(s)) - \pi_{\ep}(u_{\ep}(s)) + f(s)|_{H}^2\,ds 
    \\
    &\leq 
    4\Bigl(\int_{0}^{t}|\mu_{\ep}(s)|_{V}^2\,ds 
    + \int_{0}^{t}|\beta(u_{\ep}(s))|_{H}^2\,ds 
    + \int_{0}^{t}|\pi_{\ep}(u_{\ep}(s))|_{H}^2\,ds 
    + \int_{0}^{t}|f(s)|_{V}^2\,ds\Bigr) \\
    &\leq 
    16\Bigl(\int_{0}^{t}|\mu_{\ep}(s)|_{V}^2\,ds  
    + c_2(\ep)^2\int_{0}^{t}|u_{\ep}(s)|_{H}^2\,ds 
    + |f|_{L^2(0, T; V)}^2 \Bigr).
    \end{align*}
 Hence, by the standard elliptic regularity estimate that 
 there exists a constant $L > 0$ such that 
 $|w|_{W} \leq L|(-\Delta + I)w|_{H}$ for all $w \in W$, we infer 
    \begin{align*}
    \int_{0}^{t}|\ep u_{\ep}(s)|_{W}^2\,ds 
    &\leq 
    16L^2\Bigl(\int_{0}^{t}|\mu_{\ep}(s)|_{V}^2\,ds  
    + c_2(\ep)^2\int_{0}^{t}|u_{\ep}(s)|_{H}^2\,ds 
    + |f|_{L^2(0, T; V)}^2 \Bigr)
    \end{align*}
 for all $t\in[0, T]$.
 Therefore \eqref{epes4} follows from \eqref{epes1} and \eqref{epes2}. 
  
 Moreover, we see from \eqref{epes1} and \eqref{epes4} 
 that $u_{\ep}\in L^{\infty}(0, T; V)$ and 
 $u_{\ep}\in L^2(0, T; W)$, respectively.
 \qed
 \end{prth1.2}
 \section{Error estimates}\label{Sec5}

Regarding \ref{Pep} as approximate problems of \ref{P} 
as $\ep \searrow 0$,  we can obtain the following theorem 
which gives an information about the error estimate 
between the solution of \ref{P} and the solution of \ref{Pep}. 
Our proof is based on a direct estimate and hence 
it is simpler than that in \cite{CF-2016}.
 \begin{thm}\label{maintheorem3}
 In {\rm (C4)} and {\rm (C5)} assume further that 
 \begin{equation}\label{teisuu}
 c_2(\ep)=\tilde{c_2}\ep, \quad c_3(\ep) \equiv \tilde{c_3}
 \end{equation}
 and 
 \begin{equation}\label{shoki}
 |u_{0\ep}-u_{0}|_{V^{*}} \leq c_{4}\ep^{1/4}
 \end{equation}
 for some constants $\tilde{c_2}$, $\tilde{c_3}$ and $c_{4} > 0$ 
 independent of $\ep$. 
 Let $(u_{\ep}, \mu_{\ep})$ and $(u, \mu)$
 be weak solutions of {\rm \ref{Pep}} and {\rm \ref{P}}, 
 respectively.
 Then there exist constants $C^{*}>0$ and 
 $\overline{\ep}\in(0, 1]$, independent of $\ep$, 
 such that
     \begin{equation}\label{eres}
        |u_{\ep}-u|^2_{C([0, T]; V^{*})} + 
        \int_{0}^{T} 
        \bigl(\beta(u_{\ep}(s)) - \beta(u(s)), 
                              u_{\ep}(s) - u(s)\bigr)_{H}
        \,ds 
        \leq C^{*}\ep^{1/2}
     \end{equation}
 for all $\ep \in (0, \overline{\ep}]$.
 \end{thm}
\smallskip

 \begin{prth5.1} 
 Under the additional condition \eqref{teisuu} we have from  
 Theorems \ref{maintheorem1} and \ref{maintheorem2} 
 that 
 there exist constants $M_{1} > 0$, $M_{2} > 0$ 
 and $\overline{\ep}\in(0, 1]$, independent of $\ep$, 
 such that $2c_{1}-\tilde{c_2}\overline{\ep} \ge c_1$ and 
    \begin{align}
    &\int_{0}^{T}|u'(s)|^2_{V*}\,ds + 
     2c_{1}|u(t)|_{H}^2 \leq M_{1}, \label{apri}
    \\ 
    &\int_{0}^{T}|u_{\ep}'(s)|^2_{V*}\,ds + 
     \ep|u_{\ep}(t)|_{V}^2 + c_1|u_{\ep}(t)|_{H}^2 \leq M_{2} \label{epapri}
    \end{align}
 for all $t\in[0,T]$ and $\ep\in(0, \overline{\ep}]$. 
 Now we see from \eqref{de4} and \eqref{de7} that 
 \begin{align*} 
 \frac{1}{2}\frac{d}{ds}|u_{\ep}(s)-u(s)|_{V^{*}}^2 
 &= (u_{\ep}(s)-u(s), u_{\ep}'(s)-u'(s))_{V^{*}} \\
 &= \bigl\langle u_{\ep}(s)-u(s), 
       F^{-1}u_{\ep}'(s)-F^{-1}u'(s) \bigr\rangle_{V^{*}, V} \\
 &= -\langle u_{\ep}(s)-u(s), 
                 \mu_{\ep}(s)-\mu(s) \rangle_{V^{*}, V}.
 \end{align*}
 Here, since 
 $u_{\ep}\in L^{\infty}(0, T; V)$, 
 $u\in L^{\infty}(0, T; H)$ 
 and $\mu_{\ep}, \mu\in L^2(0, T; V)$, 
 we derive 
 \begin{align*}
 \langle u_{\ep}(s)-u(s), 
               \mu_{\ep}(s)-\mu(s) \rangle_{V^{*}, V} 
 = (u_{\ep}(s)-u(s), \mu_{\ep}(s)-\mu(s))_{H}.
 \end{align*}
 Thus, by \eqref{de8} and 
 since the function $r \mapsto \ep r + \pi_{\ep}(r)$ 
 is a monotone increasing function, it follows that 
 \begin{align*}
 &\frac{1}{2}\frac{d}{ds}|u_{\ep}(s)-u(s)|_{V^{*}}^2 
         + \bigl(u_{\ep}(s)-u(s), \beta(u_{\ep}(s))-\beta(u(s))\bigr)_{H}
 \\
 &= -\bigl(u_{\ep}(s)-u(s), 
         \ep(-\Delta+I)u_{\ep}(s) + \pi_{\ep}(u_{\ep}(s))\bigr)_{H}
 \\[1mm]
 &\leq \ep\bigl(u(s), (-\Delta+I)u_{\ep}(s)\bigr)_{H} 
          + \bigl(u(s), \pi_{\ep}(u_{\ep}(s))\bigr)_{H}.
 \end{align*}
 Integrating this inequality yields 
 \begin{align}
 &\frac{1}{2}|u_{\ep}(t)-u(t)|_{V^{*}}^2 + 
  \int_{0}^{t}
  \bigl(u_{\ep}(s)-u(s), \beta(u_{\ep}(s))-\beta(u(s))\bigr)_{H}\,ds \label{u and uep}
 \\ \notag
 &\leq \frac{1}{2}|u_{0\ep}-u_{0}|_{V^{*}}^2 + 
  \ep\int_{0}^{t} \bigl(u(s), (-\Delta + I)u_{\ep}(s)\bigr)_{H}\,ds 
  + \int_{0}^{t}\bigl(u(s), \pi_{\ep}(u_{\ep}(s))\bigr)_{H}\,ds 
 \\ \notag
 &=: A(\ep) + B_{\ep}(t) + C_{\ep}(t).
 \end{align}
 From \eqref{shoki} we have
 \begin{equation}\label{A}
 A(\ep) \leq \frac{1}{2}c_4^2\ep^{1/2}.
 \end{equation}
 By \eqref{a3ineq}, \eqref{teisuu}, \eqref{apri} and \eqref{epapri} there exist 
 a constant 
 $C_{1} > 0$ such that 
 \begin{align}
 C_{\ep}(t) \label{C} 
 \leq 
 \tilde{c_{2}}\ep\int_{0}^{t}|u_{\ep}(s)|_{H}|u(s)|_{H}\,ds 
 \leq 
 \tilde{c_{2}}\sqrt{\frac{M_{2}}{c_{1}}}\sqrt{\frac{M_{1}}{2c_{1}}}T\ep 
 \leq C_{1}\ep.
 \end{align}
 Moreover, Schwarz's inequality and \eqref{apri} give 
 \begin{align*}
 B_{\ep}(t) 
 &\leq 
 \ep^{1/2}\Bigl(\ep\int_{0}^{t}
 |(-\Delta + I)u_{\ep}(s)|_{H}^2\,ds \Bigr)^{1/2}
 \Bigl(\int_{0}^{t}|u(s)|_{H}^2\,ds \Bigr)^{1/2} 
 \\
 &\leq \sqrt{\frac{M_{1}T}{2c_{1}}}
               \Bigl(\ep\int_{0}^{t}
                |(-\Delta + I)u_{\ep}(s)|_{H}^2\,ds \Bigr)^{1/2}\ep^{1/2}
 \end{align*}
 Here, by \eqref{de7}, \eqref{de8}, Young's inequality, \eqref{Ffg},
 \eqref{apri} and \eqref{epapri}, it follows that 
 \begin{align*}
 \ep|(-\Delta + I)u_{\ep}(s)|_{H}^2 
 &= \bigl(\mu_{\ep}'(s),\ (-\Delta + I)u_{\ep}(s)\bigr)_{H} 
   -\bigl(\beta(u_{\ep}(s)),\ (-\Delta + I)u_{\ep}(s)\bigr)_{H} \\
 &\quad\ -\bigl(\pi_{\ep}(u_{\ep}(s)),\ (-\Delta + I)u_{\ep}(s)\bigr)_{H} 
   + \bigl(f(s),\ (-\Delta + I)u_{\ep}(s)\bigr)_{H} 
 \\[1mm]
 &= -\langle u_{\ep}'(s),\ u_{\ep}(s) \rangle_{V^{*}, V} 
   -\bigl(\beta(u_{\ep}(s)),\ -\Delta u_{\ep}(s)\bigr)_{H} \\
 &\quad\ -\bigl(\beta(u_{\ep}(s)),\ u_{\ep}(s)\bigr)_{H} 
   -\bigl(\pi_{\ep}(u_{\ep}(s)),\ u_{\ep}(s)\bigr)_{V} 
                                   +(g(s),\ u_{\ep}(s))_{H} \\
 &\leq -\frac{1}{2}\frac{d}{ds}|u_{\ep}(s)|_{H}^2 
       + \tilde{c_{2}}\ep|u_{\ep}(s)|_{V}^2 
       + \frac{1}{2}|g(s)|_{H}^2 
       + \frac{1}{2}|u_{\ep}(s)|_{H}^2 \\
 &\leq -\frac{1}{2}\frac{d}{ds}|u_{\ep}(s)|_{H}^2 
       + \tilde{c_{2}}M_{2} 
       + \frac{1}{2}|g(s)|_{H}^2 + \frac{M_{2}}{2c_{1}},
 \end{align*}
 and hence there exists a constant $C_{2} > 0$ 
 such that 
 \begin{align*}
 \ep\int_{0}^{t}|(-\Delta + I)u_{\ep}(s)|_{H}^2\,ds 
 \leq 
      \frac{1}{2}|u_{0\ep}|_{H}^2 
      +\tilde{c_{2}}M_{2}T 
      + \frac{1}{2}|g|_{L^2(0, T; H)}^2 
      + \frac{M_{2}}{2c_{1}}T 
 \leq C_{2}.
 \end{align*}
 Thus, there exists a constant $C_{3} > 0$ such that 
 \begin{align}\label{B}
 B_{\ep}(t) \leq C_{3}\ep^{1/2}.
 \end{align}
 Plugging \eqref{A}, \eqref{C} and \eqref{B} into \eqref{u and uep}, we have 
 \begin{align*}
 |u_{\ep}-u|_{C([0, T]; V^{*})}^2 \leq 
 c_{4}^2\ep^{1/2} + 2C_{3}\ep^{1/2} + 2C_{1}\ep, 
 \end{align*}
 and
 \begin{align*}
 \int_{0}^{T}
 \bigl(u_{\ep}(s)-u(s),\ \beta(u_{\ep}(s))-\beta(u(s))\bigr)_{H}\,ds 
 \leq \frac{1}{2}c_{4}^2\ep^{1/2} + C_{3}\ep^{1/2} + C_{1}\ep, 
 \end{align*}
 that is, there exists $C^{*} > 0$ such that 
 the error estimate \eqref{eres} holds. 
 \qed
 \end{prth5.1}
 
 \section{Examples}\label{Sec6}
 
 In this section we apply 
 Theorems \ref{maintheorem1}, \ref{maintheorem2} and \ref{maintheorem3} 
 to the following two examples.

\begin{ex1}
We consider 
$$
\beta (r) = |r|^{q-1}r + r \quad (q > 1), \qquad 
\pi_{\ep}(r) = -\ep r.           
$$         
This $\beta$ is the function obtained by adding 
the correction term $r$ to $|r|^{q-1}r$ in 
the porous media equation 
 (see, e.g., \cite{ASS-2016, M-2010, V-2007, Y-2008}). 
On the other hand, $\pi_{\ep}$ is the function 
appearing in Cahn--Hilliard type equations.
\end{ex1}
\begin{ex2}
Consider
$$
\beta(r) = |r|^{q-1}r + r \quad (0 < q < 1), \qquad 
\pi_{\ep}(r) = -\ep r
$$
This $\beta$ is the function obtained by adding 
the correction term $r$ to $|r|^{q-1}r$ in 
the fast diffusion equation 
 (see, e.g., \cite{B-1983, RV-2002, V-2007}). 
\end{ex2}

In both examples we can show that $\beta$ and $\pi_{\ep}$ 
 satisfy (C1), (C4) and (C5) as follows. 
Let $q>0$. Since 
 $$
 \beta(r) = |r|^{q-1}r + r 
 = \hat{\beta}\,'(r) 
 = \partial\hat{\beta}(r), 
 $$
 where $\hat{\beta}(r):=\frac{1}{q+1}|r|^{q+1} 
                                  + \frac{1}{2}|r|^2$, 
 we see that (C1) is satisfied. 

 Next it follows that $\pi_{\ep}(r) = -\ep r$ is 
 Lipschitz continuous and 
 \begin{align*}
 &\pi_{\ep}'(r) = -\ep, 
 \\
 &\dfrac{\ep}{2}r^2 + \hat{\pi_{\ep}}(r) 
 = \dfrac{\ep}{2}r^2 + \int_{0}^{r}\pi_{\ep}(s)\,ds 
 = \dfrac{\ep}{2}r^2 - \dfrac{\ep}{2}r^2 
 = 0. 
 \end{align*}
 Hence (C4) holds. 

 To verify (C5) 
 we assume (C3), 
 i.e., 
 $u_{0} \in L^2(\Omega)\cap L^{q+1}(\Omega)$. 
 Then we put 
 \begin{align*}
 &A_{L^2}:=-\Delta+I 
 : D(A_{L^2}):=W \subset L^2(\Omega) \to L^2(\Omega), 
 \\[1mm]
 &(J_{L^2})_{\lambda}:=(I+\lambda A_{L^2})^{-1}, 
 \\[1mm]
 &A_{L^{q+1}}:=-\Delta+I 
 : D(A_{L^{q+1}}):=Y \subset L^{q+1}(\Omega) \to L^{q+1}(\Omega), 
 \\[1mm]
 &(J_{L^{q+1}})_{\lambda}:=(I+\lambda A_{L^{q+1}})^{-1}, 
 \end{align*}
 where 
 $Y:=\bigl\{z \in W^{2,\,q+1}(\Omega)\ |\ \partial_{\nu}z = 0 
 \quad \mbox{a.e.\ on}\ \partial\Omega \bigr\}$. 
 There exists $u_{0\ep} \in W\cap Y$ such that 
    \begin{equation*}
       \begin{cases}
         u_{0\ep} + \ep(-\Delta + 1)u_{0\ep} = u_{0} 
         \quad \mbox{in}\ \Omega,
         \\[2mm]
         \partial_{\nu}u_{0\ep} = 0 
         \quad \mbox{on}\ \partial\Omega, 
       \end{cases}
    \end{equation*}
 that is, 
 $$u_{0\ep} = (J_{L^2})_{\ep}u_{0} = (J_{L^{q+1}})_{\ep}u_{0}.$$ 
 From the properties of 
 $(J_{L^2})_{\ep}$ and $(J_{L^{q+1}})_{\ep}$ we have 
  \begin{align*}
  &u_{0\ep} = (J_{L^2})_{\ep}u_{0} \to u_{0} 
  \quad \mbox{in}\ L^{2}(\Omega)\ \mbox{as}\ \ep \searrow 0, 
  \\[1mm]
  &|u_{0\ep}|_{L^2(\Omega)} 
                      = |(J_{L^2})_{\ep}u_{0}|_{L^2(\Omega)} 
  \leq |u_{0}|_{L^2(\Omega)}, 
  \\[1mm]
  &\|u_{0\ep}\|_{L^{q+1}(\Omega)} 
                = \|(J_{L^{q+1}})_{\ep}u_{0}\|_{L^{q+1}(\Omega)} 
  \leq \|u_{0}\|_{L^{q+1}(\Omega)}, 
  \end{align*}
  and hence 
  \begin{align}
  \int_{\Omega}\hat{\beta}(u_{0\ep}) \notag
  = \dfrac{1}{q+1}\|u_{0\ep}\|_{L^{q+1}(\Omega)}^{q+1} 
       + \dfrac{1}{2}|u_{0\ep}|_{L^{2}(\Omega)}^2
  \leq \dfrac{1}{q+1}\|u_{0}\|_{L^{q+1}(\Omega)}^{q+1} 
       + \dfrac{1}{2}|u_{0}|_{L^{2}(\Omega)}^2, 
  \\[2mm]
  \ep|u_{0\ep}|_{H^1(\Omega)}^2     \label{6.1}
  = \bigl(\ep(-\Delta+I)u_{0\ep}, u_{0\ep}\bigr)_{L^{2}(\Omega)} 
  = (u_{0}-u_{0\ep}, u_{0\ep})_{L^{2}(\Omega)} 
  \leq |u_{0}|_{L^{2}(\Omega)}^2. 
  \end{align}
 Hence there exists $u_{0\ep}$ satisfying (C5). 
 Moreover, we observe that 
 $$
 |u_{0\ep}-u_{0}|_{H^{-1}(\Omega)} \leq \ep^{1/2}|u_{0}|_{L^{2}(\Omega)}.
 $$ 
 Indeed, it follows from \eqref{6.1} that 
 $$
 |u_{0\ep}-u_{0}|_{H^{-1}(\Omega)}^2 
 = |\ep(-\Delta+I)u_{0\ep}|_{H^{-1}(\Omega)}^2 
 = \ep^2|Fu_{0\ep}|_{H^{-1}(\Omega)}^2 
 = \ep^2|u_{0\ep}|_{H^1(\Omega)}^2 
 \leq \ep|u_{0}|_{L^{2}(\Omega)}^2.
 $$
 Finally, letting $g \in L^2\bigl(0,T; L^2(\Omega)\bigr)$, 
 we find a function $f \in L^2(0, T; H^2(\Omega))$ satisfying (C2). 
 From the above, (C1), (C2), (C4) and (C5) hold and 
 we obtain 
 Theorems \ref{maintheorem1}, \ref{maintheorem2} and \ref{maintheorem3} 
 for the functions $\beta$ and $\pi_{\ep}$ in Examples 6.1 and 6.2.

\begin{remark}
In this paper, since the increasing condition of $\hat{\beta}$ is quadratic, we can only 
deal with 
the case of 
nondegenerate diffusion terms 
adding the correction term $u$ to $\beta(u)$. 
We can exclude such the correction term by translation with a constant 
when $\Omega$ is bounded; 
however, we cannot do it when $\Omega$ is unbounded. 
By revising the increasing condition of $\hat{\beta}$ 
with the $m$-th power ($m > 1$), we can deal with the pure porous media equation and the pure fast diffusion equation; however, it is 
delicate (cf.\ \cite{KY-2016}).
\end{remark}

%

\end{document}